\documentclass[a4paper,english,fontsize=10pt,parskip=half,abstracton]{scrartcl}
\usepackage{babel}
\usepackage[utf8]{inputenc}
\usepackage[T1]{fontenc}
\usepackage[a4paper,left=20mm,right=20mm,top=30mm,bottom=30mm]{geometry}
\usepackage{amsmath}
\usepackage{amsthm}
\usepackage{amssymb}
\usepackage{enumerate}
\usepackage{hyperref}

\newtheorem{Theorem}{Theorem} 
\newtheorem{Lemma}[Theorem]{Lemma}
\newtheorem{Proposition}[Theorem]{Proposition}
\newtheorem{Corollary}[Theorem]{Corollary}
\newtheorem{Definition}[Theorem]{Definition}

\allowdisplaybreaks[1]

\renewcommand{\phi}{\varphi}
\newcommand{\C}{\operatorname{C}}
\newcommand{\Aut}{\operatorname{Aut}}
\newcommand{\GL}{\operatorname{GL}}
\newcommand{\Irr}{\operatorname{Irr}}
\newcommand{\tr}{\operatorname{tr}}

\mathchardef\ordinarycolon\mathcode`\:  
 \mathcode`\:=\string"8000
 \begingroup \catcode`\:=\active
   \gdef:{\mathrel{\mathop\ordinarycolon}}
 \endgroup

\title{Cartan matrices and Brauer's \\$k(B)$-Conjecture III}
\author{Benjamin Sambale}
\date{\today}

\begin{document}
\frenchspacing
\maketitle
\begin{abstract}\noindent
For a block $B$ of a finite group we prove that $k(B)\le(\det C-1)/l(B)+l(B)\le\det C$ where $k(B)$ (respectively $l(B)$) is the number of irreducible ordinary (respectively Brauer) characters of $B$, and $C$ is the Cartan matrix of $B$. As an application, we show that Brauer's $k(B)$-Conjecture holds for every block with abelian defect group $D$ and inertial quotient $T$ provided there exists an element $u\in D$ such that $\C_T(u)$ acts freely on $D/\langle u\rangle$. This gives a new proof of Brauer's Conjecture for abelian defect groups of rank at most $2$. We also prove the conjecture in case $l(B)\le 3$.
\end{abstract}


\section{Introduction}
The present paper continues former work \cite{SambalekB,SambalekB2} by the author. We consider $p$-blocks $B$ of finite groups with respect to an algebraically closed field of characteristic $p$. Let $k(B)$ be the number of irreducible ordinary characters in $B$, and let $l(B)$ be the corresponding number of irreducible Brauer characters in $B$. Then the decomposition matrix $Q$ of $B$ has size $k(B)\times l(B)$ and gives a connection between the ordinary characters and the Brauer characters. It is known that $Q$ is a non-negative integral matrix such that every row contains at least one non-zero entry. On the other hand, the Cartan matrix $C$ of $B$ has a unique largest elementary divisor $p^d$ which coincides with the order of a defect group of $B$. 

The main theme of this article is the investigation of the relation between $k(B)$ and $p^d$ coming from the matrix factorization $C=Q^\textnormal{T}Q$ (here $Q^\textnormal{T}$ denotes the transpose of $Q$). This is motivated by a sixty years old conjecture by Richard Brauer~\cite{Brauerkb} which asserts that $k(B)\le p^d$. 

In the first part we study properties of $Q$ which eventually lead to an upper bound on $k(B)$ in terms of the determinant of $C$. This is of interest, since $\det C$ is determined locally via lower defect groups. As a natural next step we analyze the sharpness of this bound. Similar ideas lead to improvements of results by Olsson~\cite{OlssonIneq} and Brandt~\cite{Brandt}.
Finally, in the last section we apply these ideas to major subsections, and in particular, to blocks with abelian defect groups. Most of the notation is standard and can be found in Feit's book \cite{Feit} for instance. We denote a cyclic group of order $n$ by $Z_n$, and for convenience, $Z_n^m:=Z_n\times\ldots\times Z_n$ ($m$ factors).

\section{Determinants of Cartan matrices}

It is well known that the decomposition matrix $Q$ of a block $B$ of a finite group does not have block diagonal shape. We show that this remains true if we consider $Q$ with respect to an arbitrary basic set. This is a partial answer to a question raised in \cite{SambalekB} which suffices for our purpose. Recall that a \emph{basic set} is a basis for the $\mathbb{Z}$-module of generalized Brauer characters (see \cite[p. 148]{Feit}). The decomposition matrix with respect to a different basic set can be expressed as $QS$ where $S\in\GL(l(B),\mathbb{Z})$. 

\begin{Definition}
A matrix $Q\in\mathbb{Z}^{k\times l}$ is \emph{decomposable} if there exists a matrix $S\in\GL(l,\mathbb{Z})$ such that $QS=\bigl(\begin{smallmatrix}M_1&0\\0&M_2\end{smallmatrix}\bigr)$ where $M_1\in\mathbb{Z}^{k'\times l'}$ and $M_2\in\mathbb{Z}^{(k-k')\times (l-l')}$ for some $0<k'<k$ and $0<l'<l$. Otherwise, $Q$ is \emph{indecomposable}. 
\end{Definition}

\begin{Proposition}\label{con}
The decomposition matrix of a block of a finite group is indecomposable.
\end{Proposition}
\begin{proof}
Let $B$ be a block of a finite group with decomposition matrix $Q$ and Cartan matrix $C=Q^\textnormal{T}Q$. 
Assume that $Q$ is decomposable. Then, after changing the basic set, we may assume that $Q=\bigl(\begin{smallmatrix}Q_1&0\\0&Q_2\end{smallmatrix}\bigr)$. We consider the contribution matrix $M:=(m_{ij})=QC^{-1}Q^\textnormal{T}$ which does not depend on the basic set. Let $\chi\in\Irr(B)$ be a character of height $0$. Without loss of generality, suppose that $\chi$ corresponds to a row of $Q_1$. Choose any character $\psi\in\Irr(B)$ which corresponds to a row of $Q_2$. Since $C^{-1}$ also has block diagonal shape, $m_{\chi\psi}=0$. This contradicts \cite[Theorem~V.9.5]{Feit}.
\end{proof}

\begin{Lemma}\label{lem}
Let $Q$ be an integral $k\times l$ matrix without vanishing rows. Then $\det (Q^{\textnormal{T}}Q)=0$ or
\[\det (Q^{\textnormal{T}}Q)\ge k-l+1\]
\end{Lemma}
\begin{proof}
Let $Q=(q_{ij})$. 
By the Cauchy-Binet formula (see e.\,g. \cite[Theorem 13.8.2]{Harville}) we have
\begin{equation}\label{eq}
\det (Q^{\textnormal{T}}Q)=\sum_{\substack{V\subseteq\{1,\ldots,k\},\\|V|=l}}{(\det Q_V^\textnormal{T})}(\det Q_V)=\sum_{\substack{V\subseteq\{1,\ldots,k\},\\|V|=l}}{(\det Q_V)^2}
\end{equation}
where $Q_V:=(q_{ij}:i\in V,j=1,\ldots,l)$. We may assume that $\det Q_V\ne 1$ for some $V$, say $V=\{1,\ldots,l\}$. Now consider a row $r_j$ of $Q$ for $l<j\le k$. Suppose that $\det Q_{V'}=0$ for all $V'\subseteq\{1,\ldots,k\}$ such that $j\in V'$ and $|V'\cap\{1,\ldots,l\}|=l-1$. Then $r_j$ can be expressed by a rational linear combination of any $l-1$ rows taken from the first $l$ rows. Since the first $l$ rows of $Q$ are linearly independent, this gives the contradiction $r_j=0\in\mathbb{Z}^l$. Hence we can find a subset $V'$ as above such that $\det Q_{V'}\ne 0$. Since this can be done for every $j$ with $l<j\le k$, the claim follows from \eqref{eq}.
\end{proof}

\begin{Lemma}\label{schwer}
Let $Q\in\mathbb{Z}^{k\times l}$ be an indecomposable matrix of rank $l$ without vanishing rows. Let $C:=Q^\textnormal{T}Q$. Then 
\[\det C\ge l(k-l)+1\]
and
\[\min\{\det(C)xC^{-1}x^\textnormal{T}:0\ne x\in\mathbb{Z}^l\}\ge l.\]
\end{Lemma}
\begin{proof}
In the first step we reduce the situation such that all elementary divisors of $Q$ are $1$. Certainly, we can replace $Q$ by $QS$ where $S\in\GL(l,\mathbb{Z})$. Now assume that the greatest common divisor $d$ of the entries in the first column of $Q$ is greater than $1$. Dividing this column by $d$ gives a new matrix $\widetilde{Q}$ with the same $k$ and $l$, but $d^2\det \widetilde{C}=\det C$ where $\widetilde{C}:=\widetilde{Q}^\textnormal{T}\widetilde{Q}$. If $C^{-1}=(c_{ij})$ and $\widetilde{C}^{-1}=(\widetilde{c}_{ij})$, then $\widetilde{c}_{ij}=d^{\delta_{1i}+\delta_{1j}}c_{ij}$. 
For $x=(x_1,\ldots,x_l)\in\mathbb{Z}^l$ let $\widetilde{x}:=(x_1,dx_2,dx_3,\ldots,dx_l)$. Then
\[\det(C)xC^{-1}x^\textnormal{T}=\det(\widetilde{C})\widetilde{x}\widetilde{C}^{-1}\widetilde{x}^\textnormal{T}.\]
In particular,
\[\min\{\det(\widetilde{C})x\widetilde{C}^{-1}x^\textnormal{T}:0\ne x\in\mathbb{Z}^l\}\le \min\{\det(C)xC^{-1}x^\textnormal{T}:0\ne x\in\mathbb{Z}^l\}.\]
Hence, after replacing $Q$ by $\widetilde{Q}$ and repeating this process, we may assume that all elementary divisors of $Q$ equal $1$.

Now we argue by induction on $k$. In case $k=1$ we have $l=1$ and the result is obvious. Let $k\ge 2$.
Let $r$ be the first row of $Q$, and let $Q_1$ be the matrix obtained from $Q$ by removing $r$. If $Q_1$ has rank less than $l$, replace $Q$ by $QS$ ($S\in\GL(l,\mathbb{Z})$) such that at least one column of $Q_1$ vanishes. This means that one column of $Q$ has only one non-zero entry. Since all elementary divisors of $Q$ are $1$, this entry must be $\pm1$. But then $Q$ is decomposable. Thus, we have shown that $Q_1$ has rank $l$. We decompose $Q_1$ in the following form
\[Q_1=\begin{pmatrix}
P_1&&0\\&\ddots&\\0&&P_s
\end{pmatrix}\]
where $P_i\in\mathbb{Z}^{k_i\times l_i}$ and $\sum k_i=k-1$ and $\sum l_i=l$. Then every $P_i$ has rank $l_i$ and no vanishing rows. Moreover, we may assume that $P_i$ is indecomposable for $i=1,\ldots,s$. 
Let $C_1:=Q_1^\textnormal{T}Q_1$.
By induction we have
\[\det C_1=\prod_{i=1}^s{\det (P_i^\textnormal{T}P_i)}\ge \prod_{i=1}^s{\bigl(l_i(k_i-l_i)+1\bigr)}.\]
Moreover, a variation of Sylvester's determinant formula (see e.\,g. \cite[Theorem~18.1.1]{Harville}) shows that
\[\det C=\det(C_1+r^\textnormal{T}r)=\det C_1+\det(C_1)rC_1^{-1}r^\textnormal{T}.\]
According to the decomposition of $Q_1$, we can decompose $r=(r_1,\ldots,r_s)$ such that $r_i\in\mathbb{Z}^{l_i}$. By the hypothesis, $r$ is non-zero. Since $Q$ is indecomposable, even each $r_i$ is non-zero. By induction, $\det(P_i^\textnormal{T}P_i)r_i(P_i^\textnormal{T}P_i)^{-1}r_i^\textnormal{T}\ge l_i$. Therefore, 
\[\det C\ge\prod_{i=1}^s{\bigl(l_i(k_i-l_i)+1\bigr)}+\sum_{i=1}^s{l_i\prod_{j\ne i}{\bigl(l_j(k_j-l_j)+1\bigr)}}.\]
For any non-negative integers $\alpha_1,\ldots,\alpha_t$ we have the trivial inequality $1+\sum{\alpha_i}\le\prod{(\alpha_i+1)}$. We apply this twice and obtain
\begin{align*}
l(k-l)+1&=1+\Bigl(\sum_{i=1}^s{l_i}\Bigr)\Bigl(1+\sum_{i=1}^s{(k_i-l_i)}\Bigr)=1+\sum_{i=1}^s{l_i(k_i-l_i)}+\sum_{i=1}^s{l_i\Bigl(1+\sum_{j\ne i}{(k_j-l_j)}\Bigr)}\\
&\le \prod_{i=1}^s{\bigl(l_i(k_i-l_i)+1\bigr)}+\sum_{i=1}^s{l_i\prod_{j\ne i}{\bigl(l_j(k_j-l_j)+1\bigr)}}\le\det C.
\end{align*}
This proves the first claim.

For the second claim choose $\widetilde{x}\in\mathbb{Z}$ such that
\[m:=\det(C)\widetilde{x}C^{-1}\widetilde{x}^\textnormal{T}=\min\{\det(C)xC^{-1}x^\textnormal{T}:0\ne x\in\mathbb{Z}^l\}.\]
Obviously, the entries of $\widetilde{x}$ are coprime. It is well known that there exists a matrix $S\in\GL(l,\mathbb{Z})$ such that the first row of $S$ coincides with $\widetilde{x}$ (see e.\,g. \cite[Corollary~II.1]{Newman}). After replacing $Q$ by $QS^{-1}$, the first cofactor of $C$ coincides with $m$. Let $\widetilde{Q}$ be the matrix obtained from $Q$ by removing the first column. Then $m=\det(\widetilde{Q}^\textnormal{T}\widetilde{Q})$. Let $t$ be the number of non-zero rows of $\widetilde{Q}$.
We may assume that these are the first $t$ rows of $\widetilde{Q}$.
Suppose that $t\le l-1$. Then we can achieve as above that one column of $Q$ has only one non-zero entry. This gives a contradiction as before. Hence, $t\ge l$. For $i=1,\ldots,t$, let $\widetilde{Q}_i$ be the matrix 
consisting of the first $t$ rows of $\widetilde{Q}$ except the $i$-th row.
By the same argument as before, $\det(\widetilde{Q}_i^\textnormal{T}\widetilde{Q}_i)>0$. Hence, Lemma~\ref{lem} implies $\det(\widetilde{Q}_i^\textnormal{T}\widetilde{Q}_i)\ge t-l+1$.
Moreover, every $(l-1)\times(l-1)$ submatrix of $\widetilde{Q}_i$ shows up in exactly $t-l+1$ matrices $\widetilde{Q}_j$ (including $j=i$).
Therefore, the Cauchy-Binet formula yields
\[m=\det(\widetilde{Q}^\textnormal{T}\widetilde{Q})=\frac{1}{t-l+1}\sum_{i=1}^t{\det(\widetilde{Q}_i^\textnormal{T}\widetilde{Q}_i)}\ge t\ge l.\]
This completes the proof.
\end{proof}

Now we prove our main theorem which generalizes \cite[Theorem~1]{SambalekB} in two different directions.

\begin{Theorem}\label{main}
Let $B$ be a block of a finite group with Cartan matrix $C$. Then
\[k(B)\le\frac{\det C-1}{l(B)}+l(B)\le\det C.\]
\end{Theorem}
\begin{proof}
Let $Q\in\mathbb{Z}^{k(B)\times l(B)}$ be the decomposition matrix of $B$ such that $Q^\textnormal{T}Q=C$. By Proposition~\ref{con}, $Q$ is indecomposable. Hence, the first inequality follows from Lemma~\ref{schwer}. For the second inequality we may assume that $l(B)>1$. Then it is well known that $l(B)<k(B)$. Thus by Lemma~\ref{schwer}, $l(B)\le\det C-1$. 
Now the second inequality follows easily.
\end{proof}

Recall that a \emph{subsection} for a block $B$ of a finite group $G$ is a pair $(u,b)$ where $u\in G$ is a $p$-element and $b$ is a Brauer correspondent of $B$ in $\C_G(u)$. 
A result of Fujii~\cite[Corollary~1]{DetCartan} states that $\det C=p^d$ where $C$ is the Cartan matrix of a block $B$ with defect $d$ provided all non-trivial $B$-subsections $(u,b)$ satisfy $l(b)=1$. If this criterion holds, a result by Robinson~\cite[Theorem~3.4]{RobinsonNumber} already implies $k(B)\le p^d$. However, the condition $\det C=p^d$ is more general as one can see by the following example: Take a non-principal $3$-block of $Z_3^2\rtimes Q_8$ where the kernel of the action of $Q_8$ on $Z_3^2$ has order $2$ (see \cite[p. 40]{Kiyota}). Then $l(B)=1$, but $l(b)=2$ for a $B$-subsection $(u,b)$. 

Our next result concerns the sharpness of Theorem~\ref{main}.

\begin{Proposition}\label{detdefect}
Let $B$ be a $p$-block of a finite group with defect $d$ and Cartan matrix $C$. 
Suppose that 
\[k(B)=\frac{\det C-1}{l(B)}+l(B).\]
Then the following holds:
\begin{enumerate}[(i)]
\item $\det C=p^d$,
\item $C=(m+\delta_{ij})_{i,j}$ up to basic sets where $m:=(p^d-1)/l(B)$,
\item all irreducible characters of $B$ have height $0$.
\end{enumerate} 
\end{Proposition}
\begin{proof}
Let $l:=l(B)$, $k:=k(B)$, and let $Q=(q_{ij})$ be the decomposition matrix of $B$.
In case $l=k$ we have $k=l=1$, $p^d=1$ and the result is trivial. Thus, let $l<k$. Then in the induction step in the proof of Lemma~\ref{schwer}, we have that 
\begin{align}\label{2eqs}
1+\sum_{i=1}^s{l_i(k_i-l_i)}=\prod_{i=1}^s{\bigl(l_i(k_i-l_i)+1\bigr)}&&\text{and}&&1+\sum_{j\ne i}{(k_j-l_j)}=\prod_{j\ne i}{\bigl(l_j(k_j-l_j)+1\bigr)}
\end{align}
for each $i$.
The first equation shows that $k_i=l_i$ for all but possibly one $i$, say $i=s$. Moreover, $\det (P_i^\textnormal{T}P_i)=1$ for $i\ne s$. This implies $k_i=l_i=1$ and $P_i=(1)$ for $i\ne s$, since otherwise $P_i$ would be decomposable. Similarly, the second equation of \eqref{2eqs} gives $s=1$ or $l_s=1$.
In case $s>1$ we easily obtain
\begin{equation}\label{ind1}
Q=\begin{pmatrix}
1&\cdots&1\\
1&&0\\
&\ddots&\\
0&&1\\
&&\vdots\\
&&1
\end{pmatrix}.
\end{equation}
After replacing $Q$ by $QS$ for some $S\in\GL(l,\mathbb{Z})$ and permuting rows, we get
\begin{equation}
Q=\begin{pmatrix}\label{ind2}
1&&0\\
&\ddots&\\
0&&1\\
1&\cdots&1\\
\vdots&\ddots&\vdots\\
1&\cdots&1
\end{pmatrix},
\end{equation}
and $C=(m'+\delta_{ij})_{i,j}$ with $m':=(\det C-1)/l$. Now assume that $s=1$, i.\,e. $Q_1$ is indecomposable with the notation of the proof of Lemma~\ref{schwer}. Then $\det C_1=l(k-l-1)+1$. In case $\det C_1=1$, we must have $l=1$, $k=2$, and the claim is obvious. Therefore, we may assume that $k-l-1\ge 1$.
Moreover,
\[\det(C_1)rC_1^{-1}r^\textnormal{T}=\min\{\det(C_1)xC_1^{-1}x^\textnormal{T}:0\ne x\in\mathbb{Z}^l\}=t=l.\] 
By the last part of the proof of Lemma~\ref{schwer}, we deduce that $Q_1$ has the same shape as $Q$ in \eqref{ind1}. Hence, we may also assume that $Q_1$ is given as in \eqref{ind2}. Then one can show that $\det(C_1)C_1^{-1}=\det(C_1)1_l-(k-l-1)M$ where $1_l$ is the $l\times l$ identity matrix and all entries of $M$ are $1$. Write $r=(x_1,\ldots,x_l)$. Then
\[l=\det(C_1)rC_1^{-1}r^\textnormal{T}=\sum_{i=1}^l{x_i^2}+(k-l-1)\sum_{i<j}{(x_i-x_j)^2}.\]
Let $\alpha:=|\{i:x_i\ne 0\}|\ge 1$. Then 
\[l=\sum_{i=1}^l{x_i^2}+(k-l-1)\sum_{i<j}{(x_i-x_j)^2}\ge \alpha+(k-l-1)\alpha(l-\alpha)\ge \alpha+(k-l-1)(l-\alpha)\ge l.\]
We conclude that $r=\pm(1,\ldots,1)$ or $k=l+2$ and $x_i=\delta_{ij}$ for some fixed $j\in\{1,\ldots,l\}$. In both cases it is easy to see that $Q$ has the same shape as in \eqref{ind2}. 
Thus altogether, we have shown that $C=(m'+\delta_{ij})_{i,j}$ up to basic sets.
It follows that the first $l-1$ elementary divisors of $C$ all equal $1$. Since $p^d$ is also an elementary divisor, we obtain $\det C=p^d$ and $m'=m$.

For the last claim, note that the heights of the irreducible characters of $B$ can be read off the contribution matrix $M:=(m_{ij})=QC^{-1}Q^\textnormal{T}$ which does not depend on the chosen basic set. In the configuration described above, an easy calculation shows $\{l,p^d-m\}\ni p^dm_{ii}\not\equiv 0\pmod{p}$ for $i=1,\ldots,k$. Hence, by \cite[Theorem~V.9.4(iv)]{Feit} all heights equal $0$.
\end{proof}

By Brauer's Height Zero Conjecture, the defect groups in Proposition~\ref{detdefect} should be abelian. In fact, for any prime power $p^d>1$ and any divisor $t$ of $p^d-1$ we can construct examples as follows. Let $T\le\mathbb{F}_{p^d}^\times$ be a subgroup of order $t$ where $\mathbb{F}_{p^d}$ is the field with $p^d$ elements. Then, the principal block $B$ of $\mathbb{F}_{p^d}\rtimes T$ satisfies the hypothesis of Proposition~\ref{detdefect} with $l(B)=t$.
Also, any block with cyclic defect group satisfies the hypothesis (see \cite[Section~VII.2]{Feit}).

We may ask further, when we have equality $k(B)=\det C\ (=p^d)$ in the situation of Proposition~\ref{detdefect}. It is easy to see that in this case $l(B)\in\{1,p^d-1\}$. In both cases the defect groups must be abelian (see \cite[Proposition~1 and Theorem~3]{OkuyamaTsushima} and \cite[Theorem 7.1]{HKKS}). 

Next, we elaborate on Lemma~\ref{schwer}.

\begin{Lemma}\label{except}
Let $Q\in\mathbb{Z}^{k\times l}$ be a matrix of rank $l$ without vanishing rows. 
Suppose that for every $S\in\GL(l,\mathbb{Z})$, every column of $QS$ has at least two non-zero entries.
Then
\[\det (Q^\textnormal{T}Q)\ge l(k-l)\]
except in case
\[Q=\begin{pmatrix}
1&.&.\\1&.&.\\.&1&.\\.&1&.\\.&.&1\\.&.&1
\end{pmatrix}S\]
where $S\in\GL(3,\mathbb{Z})$ and $\det (Q^\textnormal{T}Q)=l(k-l)-1$.
\end{Lemma}
\begin{proof}
We may decompose $Q$ in the form
\[Q=\begin{pmatrix}
Q_1&&0\\&\ddots&\\0&&Q_s
\end{pmatrix}\]
where $Q_i\in\mathbb{Z}^{k_i\times l_i}$ is indecomposable, $\sum k_i=k$ and $\sum l_i=l$. 
By the hypothesis, $l_i<k_i$ for $i=1,\ldots,s$.
By Lemma~\ref{schwer}, $\det (Q_i^\textnormal{T}Q_i)\ge l_i(k_i-l_i)+1$. Hence,
\[\det (Q^\textnormal{T}Q)=\prod_{i=1}^s{\det (Q_i^\textnormal{T}Q_i)}\ge \prod_{i=1}^s{\bigl(l_i(k_i-l_i)+1\bigr)},\]
and it suffices to show
\begin{equation}\label{indclaim}
\prod_{i=1}^s{\bigl(l_i(k_i-l_i)+1\bigr)}\ge l(k-l)=\Bigl(\sum_{i=1}^s{l_i}\Bigr)\Bigl(\sum_{i=1}^s{(k_i-l_i)}\Bigr)
\end{equation}
except in case $s=3$, $l_1=l_2=l_3=1$ and $k_1=k_2=k_3=2$. We use induction on $s$. In case $s=1$ the claim is obvious. 
Now let $s\ge 2$. We may assume that $l_s\ge l_i$ for $i=1,\ldots,s$. 
By induction we have
\[\prod_{i=1}^s{\bigl(l_i(k_i-l_i)+1\bigr)}\ge\bigl(l_s(k_s-l_s)+1\bigr)\Bigl(\sum_{i=1}^{s-1}{l_i}\Bigr)\Bigl(\sum_{i=1}^{s-1}{(k_i-l_i)}\Bigr),\]
and we need to show that
\[l_s(k_s-l_s)\Bigl(\sum_{i=1}^{s-1}{l_i}\Bigr)\Bigl(\sum_{i=1}^{s-1}{(k_i-l_i)}\Bigr)\ge l_s(k_s-l_s)+ (k_s-l_s)\sum_{i=1}^{s-1}{l_i}+l_s\sum_{i=1}^{s-1}{(k_i-l_i)}.\]
Setting $\alpha:=\sum_{i=1}^{s-1}{l_i}\ge 2$ and $\beta:=\sum_{i=1}^{s-1}{(k_i-l_i)}\ge 2$, this becomes
\[(l_s\beta-1)((k_s-l_s)\alpha-1)>l_s(k_s-l_s).\]
This is true unless $l_s=k_s-l_s=1$ and $\alpha=\beta=2$. In this case we must have $s=3$, $l_1=l_2=1$ and $k_1=k_2=2$, since $l_3\ge l_i$. However, this configuration was excluded. In order to complete the proof, we have to show that this exceptional case does not interfere the induction process. For this, it suffices to consider the case $s=4$, $l_1=l_2=l_3=1$ and $k_1=k_2=k_3=2$. Then \eqref{indclaim} becomes
\[8\bigl(l_4(k_4-l_4)+1\bigr)\ge (3+l_4)(3+k_4-l_4)\]
which is equivalent to
\[l_4\bigl(7(k_4-l_4)-3\bigr)\ge3(k_4-l_4)+1.\]
This is true since $1\le l_4<k_4$.
\end{proof}

A variation of Theorem~\ref{main} generalizes a result by Brandt \cite[p. 515]{Brandt} (note that Brandt's result only applies to the exact Cartan matrix):

\begin{Proposition}
Let $B$ be a block of a finite group with Cartan matrix $C=(c_{ij})$ up to basic sets. Let $S=S_1\mathbin{\dot\cup}\ldots\mathbin{\dot\cup} S_r$ be a partition of the set $\{1,\ldots,l(B)\}$. Let $C_{S_i}:=(c_{st})_{s,t\in S_i}$ and 
\[d(S_i):=\min\Bigl\{\det(C_{S_i}),\frac{\det(C_{S_i})+1}{|S_i|}+|S_i|\Bigr\}.\] 
Then
\[k(B)\le 1-r+\sum_{i=1}^r{d(S_i)}.\]
In particular, $k(B)\le\tr C-l(B)+1$.
\end{Proposition}
\begin{proof}
We may assume that $k(B)>1$.
Let $Q=(q_{ij})$ be the decomposition matrix of $B$, and let $Q_{S_i}:=(q_{st})_{t\in S_i}$ for $i=1,\ldots,r$. 
For $i=1,\ldots,\alpha$, let $P_i$ be the matrix obtained from $Q_{S_1}$ by removing the $i$-th row. In case $\det P_i^\textnormal{T}P_i=0$ we can achieve as usual that $Q$ has one column with only one non-zero row. By the orthogonality relations, $B$ contains an irreducible character which vanishes on the $p$-singular elements. However, this contradicts $k(B)>1$. Hence, by Lemma~\ref{lem}, $\det (P_i^\textnormal{T}P_i)\ge \alpha-|S_i|$. Now an application of the Cauchy-Binet formula as in the proof of Lemma~\ref{schwer} shows that $\alpha\le\det C_{S_1}$. Moreover, by Lemma~\ref{except}, $\alpha\le(\det C_{S_1}+1)/|S_1|+|S_1|$. Thus altogether, $\alpha\le d(S_1)$.
By Proposition~\ref{con}, we may assume that $Q_{S_1}^\textnormal{T}Q_{S_2}\ne 0$ after permuting the columns of $Q$ if necessary. Hence, in the worst case, the non-zero rows of $Q_{S_2}$ can only contribute $d(S_2)-1$ new non-zero rows of $Q$. Continuing this process leads to the first claim. The last claim follows by taking $S_i=\{i\}$ for $i=1,\ldots,l(B)$.
\end{proof}

Similar inequalities were given in \cite{KuelshammerWada}.

The next result concerns the open question raised in \cite{SambalekB}. The number of irreducible characters of height $0$ of $B$ is denoted by $k_0(B)$.

\begin{Proposition}
Let $B$ be a $p$-block of a finite group with defect $d$ and Cartan matrix $C$. Suppose that there exists a basic set such that
\[C=\begin{pmatrix}
C_1&0\\0&C_2
\end{pmatrix}\]
where $p^d$ occurs as elementary divisor of $C_1\in\mathbb{Z}^{l_1\times l_1}$. Then
\[k_0(B)\le\frac{\det C_1-1}{l_1}+l_1.\]
\end{Proposition}
\begin{proof}
Let $Q_1$ be the part of the decomposition matrix of $B$ such that $Q_1^\textnormal{T}Q_1=C_1$.
Since all elementary divisors of $C_2$ are strictly smaller than $p^d$, we see that all entries of $p^dC_2^{-1}$ are divisible by $p$. Let $\chi\in\Irr(B)$ be a character whose corresponding row in $Q_1$ is zero. Then it follows easily that the contribution $p^dm_{\chi\chi}$ is divisible by $p$. In particular, $\chi$ has positive height. Therefore, $k_0(B)$ is at most the number of non-zero rows of $Q_1$. If $Q_1$ is indecomposable, the claim follows from Lemma~\ref{schwer}. Now assume that $Q_1$ is decomposable. Then by Lemma~\ref{except}, $Q_1$ has at most $(\det C_1)/l_1+l_1$ non-zero rows (observe that the exceptional case cannot occur). However, we may decompose $Q_1$ and replace $C_1$ by the corresponding smaller matrix. Then the new matrix $Q_1$ has at most $(\det C_1)/l_1+l_1-1\le(\det C_1-1)/l_1+l_1$ non-zero rows. This completes the proof.
\end{proof}

Our next result extends a theorem by Olsson \cite[Corollary~7]{Olsson}. The proof (following \cite[Theorem~1]{SambalekB}) makes use of the reduction theory of quadratic forms in the sense of Minkowski. However, we will not refer to the precise definition of a reduced form. Nevertheless, recall (see \cite[p. 396]{Conway}) that a reduced quadratic form corresponding to a symmetric matrix $(\alpha_{ij})\in\mathbb{Z}^{l\times l}$ satisfies 
\begin{align*}
\alpha_{11}\le\ldots&\le \alpha_{ll},&\\
2|\alpha_{ij}|&\le\alpha_{ii}&(i&<j),\\
2|\alpha_{ij}\pm\alpha_{ik}\pm\alpha_{jk}|&\le\alpha_{ii}+\alpha_{jj}&(i&<j<k).
\end{align*}

\begin{Proposition}
Let $B$ be a $p$-block with defect $d$ and $l(B)\le 3$. Then $k(B)\le p^d$.
\end{Proposition}
\begin{proof}
By \cite[Corollaries~5 and 7]{OlssonIneq} we may assume that $l(B)=3$ and $p\ge 3$. Let $C=(c_{ij})$ be the Cartan matrix of $B$ (up to basic sets). Let $p^e\le p^f<p^d$ be the elementary divisors of $C$. We consider $\widetilde{C}:=p^dC^{-1}$. By \cite[Corollary~2.5]{RobinsonNumber}, we may assume that $\widetilde{C}$, considered as a quadratic form, does not represent the number $1$. On the other hand, by \cite[Theorem V.9.17]{Feit}, we may assume that $\widetilde{C}$ represents $2$. Thus, after changing the basis set, we may assume that the first cofactor of $C$ is $c_{22}c_{33}-c_{23}^2=2p^{e+f}$. 
Let $\overline{c}_{ij}:=3^{-e}c_{ij}\in\mathbb{Z}$. 
We may assume that the matrix $\bigl(\begin{smallmatrix}
\overline{c}_{22}&\overline{c}_{23}\\\overline{c}_{23}&\overline{c}_{33}
\end{smallmatrix}\bigr)$ is reduced as quadratic form (this will not change the cofactor).
By \cite{Barnes}, we have $4\overline{c}_{22}\overline{c}_{33}-\overline{c}_{22}^2=3\overline{c}_{22}\overline{c}_{33}+\overline{c}_{22}(\overline{c}_{33}-\overline{c}_{22})\le 8p^{f-e}$ and
\[\overline{c}_{22}+\overline{c}_{33}\le \frac{5}{4}\overline{c}_{22}+\frac{2p^{f-e}}{\overline{c}_{22}}=:f(\overline{c}_{22}).\]
Moreover, $\overline{c}_{22}\le 2\sqrt{2p^{f-e}/3}$. Now $f$ is a convex function on the interval $[1,2\sqrt{2p^{f-e}/3}]$ which assumes its maximum on one of the two borders. One can show that $f(1)\ge f(2\sqrt{2p^{f-e}/3})$ unless $p^{f-e}=1$ (and then $\overline{c}_{22}=1$). 
Therefore in any case,
\[c_{22}+c_{33}\le p^e\lfloor f(1)\rfloor=p^e+2p^f.\]
Let $Q$ be the decomposition matrix of $B$, and let $\alpha$ be the number of rows of $Q$ of the form $(*,0,0)$. It is easy to see that
\[k(B)\le \alpha+c_{22}+c_{33}\le \alpha+p^e+2p^f.\]
By \cite[Proposition~2.2]{Plesken} the matrix $p^d1_{k(B)}-Q\widetilde{C}Q^\textnormal{T}$ is positive semidefinite. Let $Q_1$ be the submatrix of $Q$ consisting only of the rows of type $(*,0,0)$. Then also $L:=p^d1_\alpha-Q_1\widetilde{C}Q_1^\textnormal{T}=p^d1_\alpha-2Q_1Q_1^\textnormal{T}$ is positive semidefinite. In particular, 
\[0\le\det L=\det(1_3-2p^{-d}Q_1^\textnormal{T}Q_1)p^{\alpha d}\le (1-2p^{-d}\alpha)p^{\alpha d}\]
and $\alpha\le p^d/2$ (see \cite[Theorem~18.1.1]{Harville}).
Hence, we have proved that
\[k(B)\le \frac{p^d}{2}+p^e+2p^f\le \frac{p^d}{2}+3p^f.\] 
In order to show $k(B)\le p^d$ it suffices to handle the cases $p=3$, $f=d-1$ and $p=5$, $e=f=d-1$. We consider the latter case first. Then $\widetilde{C}$ has elementary divisors $1$, $5$, $5$. This allows only finitely many choices for $\widetilde{C}$ up to basic sets. By the Brandt-Intrau-Schiemann tables \cite{Nebe} it follows that
\begin{align*}
\widetilde{C}=\begin{pmatrix}
2&1&0\\1&3&0\\0&0&5
\end{pmatrix}&&\text{and}&&C=5^{d-1}\begin{pmatrix}
3&-1&0\\-1&2&0\\0&0&1
\end{pmatrix}.
\end{align*}
In this case the claim follows by \cite[Theorem~A]{KuelshammerWada}. 

Finally, suppose that $p=3$ and $f=d-1$. Then $\widetilde{C}$ has elementary divisors $1$, $3$, $3^{d-e}$. 
After replacing the basic set if necessary, we may assume that
\[\widetilde{C}=\begin{pmatrix}
2&1&\epsilon\\
1&a&b\\\epsilon&b&c
\end{pmatrix}.\]
with $\epsilon\in\{0,1\}$ and $2|b|\le\min\{a,c\}$ (but not necessarily $a\le c$). 
Assume first that $b=0$. Since the greatest common divisor of all the $2\times 2$ minors of $\widetilde{C}$ equals $3$ (see \cite[Theorem~9.64]{Rotman}), we get $\epsilon=0$. Then, $a\in\{2,(3^{d-e}-1)/2\}$. In the second case, the claim follows from \cite[Theorem~A]{KuelshammerWada}. Hence, we may assume that
\begin{align*}
\widetilde{C}=\begin{pmatrix}
2&1&.\\
1&2&.\\
.&.&3^{d-e}
\end{pmatrix}&&\text{and}&&
C=3^e\begin{pmatrix}
2\cdot 3^{d-e-1}&-3^{d-e-1}&.\\
-3^{d-e-1}&2\cdot 3^{d-e-1}&.\\
.&.&1
\end{pmatrix}.
\end{align*}
If no row of the decomposition matrix of $B$ has type $(0,0,*)$, then we are done by \cite{KuelshammerWada}. Hence, let $\chi\in\Irr(B)$ whose corresponding row has the form $q_\chi=(0,0,*)$. Let $M:=(m_{ij})=Q\widetilde{C}Q^\textnormal{T}$ be the contribution matrix of $B$ (strictly speaking, multiplied by $p^d$). Since $\tr M=3^dl(B)=3^{d+1}$ (see \cite[Theorem~V.9.4(iii)]{Feit}), we may assume that there is a row $q_\psi$ of $Q$ ($\psi\in\Irr(B)$) such that $q_\psi\widetilde{C}q_\psi^\textnormal{T}=2$. It is easy to see that $q_\psi$ has the form $q_\psi=(*,*,0)$. This implies $m_{\chi\psi}=0$, and $\psi$ has positive height by \cite[Theorem~V.9.5]{Feit}. However, this gives the contradiction $q_\psi\widetilde{C}q_\psi^\textnormal{T}\ge 9$ (see \cite[Theorem~V.9.4(iv)]{Feit}).

Therefore, we are left with the case $b\ne 0$. Here, by the Brandt-Intrau-Schiemann tables we may assume that $\det\widetilde{C}\ge 81$, i.\,e. $e+1<f=d-1$. 
Since the greatest common divisor of all the $2\times 2$ minors of $\widetilde{C}$ equals $3$, we get $|2b-\epsilon|\ge 3$.
Hence, there exists a sign $\delta=\pm1$ such that $|1+\epsilon+\delta b|\ge 3$. The reduction theory gives $a,c\ge 4$ (observe that apart from interchanging $a$ and $c$, we may assume that $\widetilde{C}$ is reduced). 
Moreover, $2a-1\equiv 0\pmod{3}$ and $a\ge 5$. Similarly, $c\ge 5$. By \cite{Barnes}, $ac\le\det\widetilde{C}=3^{d-e+1}$. This shows
\[a+c\le 5+\frac{3^{d-e+1}}{5}.\]
For the entries of $C$ we get $c_{22}=3^{e-1}(2c-\epsilon^2)\le3^{e-1}\cdot 2c$, $c_{33}=3^{e-1}(2a-1)$ and $|c_{23}|=3^{e-1}|2b-\epsilon|\ge 3^e$. Hence,
\[c_{22}+c_{33}-|c_{23}|\le 3^{e-1}\Bigl(6+\frac{2}{5}3^{d-e+1}\Bigr).\]
Since $3^{d-e+1}\ge 81$, we deduce $c_{22}+c_{33}-|c_{23}|\le 3^d/2$.
Now the argument in the first part of the proof in combination with \cite{KuelshammerWada} yields
$k(B)\le\alpha+c_{22}+c_{33}-|c_{23}|\le 3^d$, and we are done.
\end{proof}

Let $B$ be a counterexample for Brauer's $k(B)$-Conjecture. Then the decomposition matrix $Q$ of $B$ fulfills the following properties:
\begin{enumerate}[(i)]
\item $Q\in\mathbb{Z}^{k\times l}_{\ge0}$, 
\item no row of $Q$ vanishes,
\item all elementary divisors of $Q$ equal $1$,
\item $Q$ is indecomposable,
\item $Q^\textnormal{T}Q$ has a unique largest elementary divisor which is a power of a prime $p$, say $p^d$,
\item every diagonal entry of $(m_{ij}):=p^dQ(Q^\textnormal{T}Q)^{-1}Q^\textnormal{T}$ is either divisible by $p^2$ or not divisible by $p$,
\item if $m_{ij}=0$, then $p^2\mid m_{ii}$ and $p^2\mid m_{jj}$,
\item $k>p^d$.
\end{enumerate}

We were unable to find any matrix $Q$ with these constraints. Thus, there is some hope that Brauer's Conjecture follows from matrix theory.

\section{Major subsections}
In this section we replace the Cartan matrix of a block $B$ by the Cartan matrix of a major $B$-subsection. Recall that a $B$-subsection $(u,b)$ is \emph{major} if $b$ and $B$ have the same defect.
This is always the case for blocks with abelian defect groups.

\begin{Theorem}\label{majorsub}
Let $B$ be a $p$-block of a finite group $G$. Let $(u,b)$ be a major $B$-subsection such that $|\langle u\rangle|=p^r$ and $b$ has defect $d$ and Cartan matrix $C$. If $\det(p^{-r}C)=p^{d-r}$, then Brauer's $k(B)$-Conjecture holds for $B$.
\end{Theorem}
\begin{proof}
It is well known that $b$ dominates a block $\overline{b}$ of $\C_G(u)/\langle u\rangle$ with defect $d-r$ and Cartan matrix $\overline{C}:=p^{-r}C$. 
Let $\overline{Q}$ be the decomposition matrix of $\overline{b}$. Then by Proposition~\ref{con}, $\overline{Q}$ is indecomposable. Therefore, Lemma~\ref{schwer} implies
\[\min\{p^dxC^{-1}x^{\textnormal{T}}:0\ne x\in\mathbb{Z}^l\}=\min\{\det(\overline{C})x\overline{C}^{-1}x^{\textnormal{T}}:0\ne x\in\mathbb{Z}^l\}\ge l(\overline{b})=l(b).\]
Now the claim follows from a result by Brauer (see \cite[Theorem~V.9.17]{Feit}).
\end{proof}

Theorem~\ref{majorsub} generalizes Brauer's argument for the case where $\overline{b}$ has cyclic defect group (see \cite[Lemma~VII.10.11]{Feit}). In general it is not true that
\[m=\min\{xp^dC^{-1}x^{\textnormal{T}}:0\ne x\in\mathbb{Z}^{l(B)}\}\ge l(B)\] 
for every block $B$ with defect $d$ and Cartan matrix $C$. A counterexample is given by the principal $2$-block of $Z_2^3\rtimes (Z_7\rtimes Z_3)$. Here, $m=4<5=l(B)$. This answers a question by Olsson (see \cite[Remark~G]{OlssonIneq}).

\begin{Definition}
A group $G$ acts \emph{freely} on a group $H$ if $G\le\Aut(H)$ and $\C_G(h)=1$ for all $1\ne h\in H$ \textup{(}i.\,e. $H\rtimes G$ is a Frobenius group whenever $G\ne 1$\textup{)}.
\end{Definition}

\begin{Corollary}\label{cor}
Let $B$ be a block with abelian defect group $D$ and inertial quotient $T$. Suppose that there exists an element $u\in D$ such that $\C_T(u)$ acts freely on $D/\langle u\rangle$. Then $k(B)\le |D|$.
\end{Corollary}
\begin{proof}
We consider a $B$-subsection $(u,b)$. Then $b$ has defect group $D$ and inertial quotient $\C_T(u)$. As usual $b$ dominates a block $\overline{b}$ with defect group $D/\langle u\rangle$ and inertial quotient $\C_T(u)$.
By the hypothesis, all non-trivial $\overline{b}$-subsections $(v,\beta)$ have inertial index $1$. In particular, $l(\beta)=1$. By a result by Fujii~\cite[Corollary~1]{DetCartan}, it follows that $\det \overline{C}=|D/\langle u\rangle|$ where $\overline{C}$ is the Cartan matrix of $\overline{b}$. Now Theorem~\ref{majorsub} implies the claim.
\end{proof}

The condition in Corollary~\ref{cor} is equivalent to $\C_T(u)\cap\C_T(v)=1$ for all $v\in D\setminus\langle u\rangle$. By a result of Halasi and Podoski~\cite[Corollary~1.2]{base2}, it is known that there are always \emph{some} elements $u,v\in D$ such that $\C_T(u)\cap\C_T(v)=1$.

For abelian defect groups, Corollary~\ref{cor} is all what one can expect to deduce from Theorem~\ref{majorsub}. This can be seen from the following example: Let $D$ be an abelian $p$-group and let $T\le\Aut(D)$ be a $p'$-group which does not act freely on $D$. Then by \cite[Theorem~IV.3.11]{Feit}, the Cartan matrix $C$ of the principal $p$-block of $D\rtimes T$ satisfies $\det C>|D|$.

We remark also that, by \cite[Theorem 1.2]{FrobeniusInertial}, the stable center of the block $\overline{b_u}$ in the proof of Corollary~\ref{cor} is a symmetric algebra.

\begin{Corollary}[Brauer, see {\cite[Theorem~VII.10.13]{Feit}}]\label{cor2}
Let $B$ be a block with abelian defect group of rank at most $2$. Then Brauer's $k(B)$-Conjecture holds for $B$.
\end{Corollary}
\begin{proof}
Let $D$ be a defect group of $B$, and let $T$ be the inertial quotient of $B$. We fix an element $u\in D$ of maximal order. Then $\C_T(u)$ acts freely on the cyclic group $D/\langle u\rangle$. Hence, the claim follows from Corollary~\ref{cor}.
\end{proof}

Compared to Brauer's original proof, the proof of Corollary~\ref{cor2} does not depend on Dade's deep theory of cyclic defect groups.

As another application of Corollary~\ref{cor} we give a more concrete example: Let $B$ be a block with defect group $Z_2^7$ and inertial quotient $Z_{127}\rtimes Z_7$. Then Brauer's $k(B)$-Conjecture does not follow from previous results by the present author in \cite{Sbrauerfeit}. However, Corollary~\ref{cor} applies in this situation.

We use the opportunity to provide a dual version of \cite[Lemma~5]{Sbrauerfeit} which makes use of a recent result by Keller-Yang~\cite{KellerYang}.

\begin{Proposition}
Let $B$ be a block with abelian defect group $D$ and inertial quotient $T$. If $|T'|\le 4$, then Brauer's $k(B)$-Conjecture holds for $B$.
\end{Proposition}
\begin{proof}
As usual, the action of $T$ on $D$ is faithful and coprime. By \cite[Theorem~1.1]{KellerYang} there exists an element $u\in D$ such that $\lvert\C_T(u)\rvert\le 4$. Now the claim follows from \cite[Lemma~4]{Sbrauerfeit}.
\end{proof}

\section*{Acknowledgment}
This work is supported by the Carl Zeiss Foundation and the Daimler and Benz Foundation. The author thanks Gabriele Nebe for pointing out a Magma implementation of Kneser's algorithm for quadratic forms.

\begin{center}
Benjamin Sambale\\
Institut für Mathematik\\
Friedrich-Schiller-Universität\\
07743 Jena\\
Germany\\
\href{mailto:benjamin.sambale@uni-jena.de}{benjamin.sambale@uni-jena.de}
\end{center}
\end{document}